\documentclass[12pt,reqno]{amsproc}
\usepackage{amsmath,amssymb,amsthm,mathtools,enumerate,wasysym,times}
\calclayout
\theoremstyle{plain}
\newtheorem{theorem}{Theorem}[section]
\newtheorem{proposition}[theorem]{Proposition}
\newtheorem{problem}[theorem]{Problem}
\newtheorem{conjecture}[theorem]{Conjecture}
\newtheorem{lemma}[theorem]{Lemma}
\theoremstyle{definition}
\newtheorem{definition}[theorem]{Definition}
\newtheorem{remark}[theorem]{Remark}
\newcommand{\emp}{\varnothing}
\newcommand{\ba}{\boldsymbol{\alpha}}
\newcommand{\cA}{\mathcal{A}}
\newcommand{\rc}{\mathrm{c}}

\newcommand{\bk}{\boldsymbol{k}}
\newcommand{\PP}{\mathcal{P}}
\newcommand{\QQ}{\mathbb{Q}}
\newcommand{\RR}{\mathbb{R}}
\newcommand{\ZZ}{\mathbb{Z}}
\newcommand{\fZ}{\mathfrak{Z}}
\newcommand{\cZ}{\mathcal{Z}}
\numberwithin{equation}{section}
\address{Department of Mathematical Sciences, Aoyama Gakuin University, 5-10-1 Fuchinobe, Chuo-ku, Sagamihara-shi, Kanagawa, 252-5258, Japan}
\email{seki@math.aoyama.ac.jp}
\title[]{Regular primes, non-Wieferich primes, and finite multiple zeta values of level $N$}
\author{Shin-ichiro Seki}
\date{}
\thanks{This research was supported by JSPS KAKENHI Grant Number 21K13762.}
\begin{document}
\maketitle
%%%%%%%%%%%%%%%%%%%%%%%%%%%%%%%%%%%%%%%%%%%%%%%%%%
\begin{abstract}
We introduce finite multiple zeta values of general levels and discuss the relationship between the non-zeroness of these values and regular or non-Wieferich primes.
Because it is challenging to prove the infinitude of these kinds of primes, we suggest tackling several related problems more promptly.
\end{abstract}
%%%%%%%%%%%%%%%%%%%%%%%%%%%%%%%%%%%%%%%%%%%%%%%%%%
\section{Introduction}\label{sec:Intro}
The following two old conjectures about prime numbers are still open.
\begin{conjecture}\label{conj:regular}
There exist infinitely many regular primes.
\end{conjecture}
\begin{conjecture}\label{conj:non-Wieferich}
There exist infinitely many non-Wieferich primes.
\end{conjecture}
For a prime $p$, we call $p$ \emph{regular} when $p$ is odd and the class number of the $p$th cyclotomic field is not divisible by $p$.
We call $p$ \emph{non-Wieferich} when $2^{p-1}-1$ is not divisible by $p^2$.
These kinds of primes are each associated with the following theorems related to Fermat's last theorem.
\begin{theorem}[Kummer~\cite{Kummer}]
If $p$ is a regular prime, then Fermat's last theorem for the exponent $p$ is correct.
\end{theorem}
\begin{theorem}[Wieferich~\cite{Wieferich}]
If $p$ is a non-Wieferich prime, then the first case of Fermat's last theorem for the exponent $p$ is correct.
\end{theorem}
In addition to the pure interest in the distribution of primes, historically, one of the motivations to resolve Conjectures~\ref{conj:regular} and \ref{conj:non-Wieferich} can be attributed to Fermat's last theorem.
If Conjecture~\ref{conj:non-Wieferich} is true, then it ensures that the first case of Fermat's last theorem holds for infinitely many prime exponents.
Similarly, if Conjecture~\ref{conj:regular} is true, then it ensures that Fermat's last theorem holds for infinitely many prime exponents.
It should be noted that for the first case, it has been shown by Adleman--Heath-Brown~\cite{Adleman-Heath-Brown} to hold for infinitely many prime exponents without relying on Conjecture~\ref{conj:non-Wieferich}.
As is well known, since Fermat's last theorem was fully resolved by Wiles~\cite{Wiles}, this motivation can be said to have vanished.

In this decade, there has been a focus on the study of finite multiple zeta values and Conjecture~\ref{conj:regular} is related to the non-vanishing nature of these values.
Furthermore, in this note, we define a generalization of finite multiple zeta values to general levels and explore their relationship with Conjecture~\ref{conj:non-Wieferich}.
From the study of these values, a new and compelling motivation has arisen to resolve Conjectures~\ref{conj:regular} and \ref{conj:non-Wieferich} after the work by Wiles.
In Section~\ref{sec:problems}, we discuss problems that are believed to be resolved before these conjectures.
%%%%%%%%%%%%%%%%%%%%%%%%%%%%%%%%%%%%%%%%
\section{Finite multiple zeta values of level $N$}\label{sec:FMZVL}
\subsection{Finite multiple zeta values}
The finite multiple zeta values, which are analogues of multiple zeta values (periods of mixed Tate motives over $\ZZ$), were defined as elements of the ring 
\[
\cA\coloneqq\left.\left(\prod_{p\in\PP}\ZZ/p\ZZ\right) \right/ \left(\bigoplus_{p\in\PP}\ZZ/p\ZZ\right)
\]
by Kaneko--Zagier.
Here, $\PP$ denotes the set of all prime numbers.
This ring $\cA$, termed the ring of integers modulo infinitely large primes, is not an integral domain but is a reduced ring.
Through the diagonal embedding, $\cA$ possesses a $\QQ$-algebra structure.
Elements of $\mathcal{A}$ are actually determined as equivalence classes, but they are often represented by their representative elements $(a_p)_{p\in\PP}$ in $\prod_{p\in\PP}\ZZ/p\ZZ$.
In this context, even if $a_p$ is not defined for a finite number of primes $p$, it remains well-defined as an element of $\mathcal{A}$ and poses no problem.
We call a tuple of positive integers $\bk=(k_1,\dots,k_r)$ an \emph{index}.
Its \emph{depth} is given by $r$, and its \emph{weight} by $k_1+\cdots+k_r$.
We consider the empty set $\emp$ as an index with both depth and weight being $0$, and refer to it as the \emph{empty index}.
For each prime $p$ and non-empty index $\bk=(k_1,\dots, k_r)$, we define $\zeta_p(\bk)\in\ZZ_{(p)}$ by
\[
\zeta_{p}(\bk)\coloneqq\sum_{0<m_1<\cdots<m_r<p}\frac{1}{m_1^{k_1}\cdots m_r^{k_r}}.
\]
Using the natural isomorphism $\ZZ_{(p)}/p\ZZ_{(p)}\simeq\ZZ/p\ZZ$, we then define the \emph{finite multiple zeta value}  $\zeta^{}_{\cA}(\bk)\in\cA$ by $\zeta^{}_{\cA}(\bk)\coloneqq(\zeta_{p}(\bk)\bmod{p})_{p\in\PP}$.
Put $\zeta^{}_{\cA}(\emp)\coloneqq1$.
There are many $\QQ$-linear relations among finite multiple zeta values.
For example, the relation
\begin{equation}\label{eq:example-relation}
2\zeta^{}_{\cA}(1,2,3)+\zeta^{}_{\cA}(1,2,1,2)+\zeta^{}_{\cA}(1,2,2,1)+\zeta^{}_{\cA}(1,1,1,3)=0
\end{equation}
holds, as explained in \cite[Section~3.7]{Saito}.
For an index $\bk$ of weight $k$, we shall refer to $\zeta^{}_{\cA}(\bk)$ as the finite multiple zeta value of weight $k$. 
For a non-negative integer $k$, we denote the $\QQ$-vector space spanned by all finite multiple zeta values of weight $k$ as $\cZ_{\cA,k}$, and set $\cZ_{\cA} \coloneqq \sum_{k=0}^{\infty} \cZ_{\cA,k}$.
By the so-called harmonic (stuffle) product formula, it is understood that $\cZ_{\cA}$ is a $\QQ$-subalgebra of $\cA$.
Research on the algebraic structure of $\cZ$, the $\QQ$-subalgebra of $\RR$ spanned by all multiple zeta values, has been ongoing for many years.
Similarly, the investigation of the algebraic structure of $\cZ_{\cA}$ is both fascinating and extensive.
Furthermore, a surprising algebraic isomorphism $\cZ_{\cA} \simeq \cZ/\zeta(2)\cZ$ is conjectured to exist, which is known as the Kaneko--Zagier conjecture.
This implies a connection between ``prime numbers'' and ``zeta'' in a form entirely different from the Euler product formula, making it a highly intriguing conjecture.
For a detailed correspondence, refer to \cite[Conjecture~9.5]{Kaneko}.
\subsection{Finite multiple zeta values of level $N$}
Recently, the finite multiple zeta value of level two, $\zeta_{\cA}^{(2)}(\bk)$, was introduced by Kaneko, Murakami, and Yoshihara in \cite{Kaneko-Murakami-Yoshihara}.
In this context, we define finite multiple zeta values of higher levels.
Let $N$ be a positive integer.
Elements of $\ZZ/N\ZZ$ are considered as equivalence classes with respect to the equivalence relation of congruence modulo $N$ in the usual way.
Furthermore, the equivalence class to which an integer $a$ belongs will be denoted by $\overline{a}$.
For each prime $p$, non-empty index $\bk=(k_1,\dots,k_r)$, and $\ba\coloneqq(\alpha_1,\dots,\alpha_r)\in(\ZZ/N\ZZ)^r$, we define $\zeta_{p,N}^{\ba}(\bk)\in\ZZ_{(p)}$ as follows:
\[
\zeta_{p,N}^{\ba}(\bk)\coloneqq\sum_{\substack{0<m_1<\cdots<m_r < p\\ m_i\in\alpha_i \text{ for all } 1\leq i\leq r}}\frac{1}{m_1^{k_1}\cdots m_r^{k_r}}.
\]
For the case $\bk=\emp$, there is only one element of $(\ZZ/N\ZZ)^r$ ($r=0$), which we denote as $\bullet$.
For convenience, we set $\zeta_{p,N}^{\bullet}(\emp)\coloneqq 1$ and for every index $\bk$, $\zeta_{p,N}^{\boxtimes}(\bk)\coloneqq 0$.
\begin{definition}
We call a map $c\colon(\ZZ/N\ZZ)^{\times}\to(\ZZ/N\ZZ)^r\cup\{\boxtimes\}$ a \emph{color map}.
We define $\zeta_{\cA,N}^{c}(\bk)\in\cA$, referred to as a \emph{finite multiple zeta value of level $N$ with color map $c$}, by $\zeta_{\cA,N}^{c}(\bk)\coloneqq(\zeta_{p,N}^{c(\overline{p})}(\bk)\bmod{p})_{p\in\PP}$.
\end{definition}
Note that since there are only finitely many prime numbers $p$ for which $c(\overline{p})$ is not defined, $\zeta_{\cA,N}^{c}(\bk)\in\cA$ is well-defined.
These values are finite analogues of the multiple zeta values of level $N$ with color defined by Yuan--Zhao in \cite[Section~2]{Yuan-Zhao}.
\begin{remark}
After the original draft of this paper was completed, Masataka Ono kindly informed the author of analogues of finite multiple zeta values defined by Tasaka using $\zeta_{p,N}^{\ba}(\bk)$ in \cite[Section~6.4]{Tasaka}.
While these values are defined in modified rings, not in the ring $\cA$, they differ from ours but are closely related.
Regardless, further research on these values is expected in the future.
\end{remark}
For later use, we give the following definition here: for each integer $j$ satisfying $0\leq j<N$, we define a map $[j]$ by $[j](\alpha)\coloneqq(-j\alpha,\dots,-j\alpha)\in(\ZZ/N\ZZ)^r$ for any $\alpha\in(\ZZ/N\ZZ)^{\times}$.
Although this definition depends on both $N$ and $r$, it poses no issues when used in contexts where $N$ and $r$ are specified, as in $\zeta_{\cA,N}^{[j]}(\bk)$.
When $r=0$, $[j](\alpha)=\bullet$.
For an index $\bk$ of weight $k$, let $\zeta_{\cA}^{(N)}(\bk)\coloneqq N^k\zeta_{\cA,N}^{[0]}(\bk)$.
Then, we have $\zeta_{\cA}^{(1)}(\bk)=\zeta_{\cA}^{}(\bk)$ and $\zeta_{\cA}^{(2)}(\bk)$ coincides with the value defined in \cite{Kaneko-Murakami-Yoshihara}.

For each non-negative integer $k$, we define a $\QQ$-vector space $\cZ_{\cA,k}(N)$ by
\[
\cZ_{\cA,k}(N)\coloneqq\sum_{\bk: \text{ index of weight }k}\sum_{c\colon(\ZZ/N\ZZ)^{\times}\to(\ZZ/N\ZZ)^{\text{depth of }\bk}\cup\{\boxtimes\}}\QQ\cdot\zeta_{\cA,N}^{c}(\bk).
\]
We set $\cZ_{\cA}(N)\coloneqq\sum_{k=0}^{\infty}\cZ_{\cA,k}(N)$.
This $\cZ_{\cA}(N)$ is the space spanned by all finite multiple zeta values of level $N$.
A natural generalization of the harmonic product formula holds, and it should be almost obvious that $\cZ_{\cA}(N)$ possesses the structure of a $\QQ$-subalgebra of $\cA$ such that $\cZ_{\cA,k_1}(N)\cdot\cZ_{\cA,k_2}(N)\subset\cZ_{\cA,k_1+k_2}(N)$.
For example, we have
\begin{multline*}
\zeta_{\cA,N}^f(k_1,k_2)\zeta_{\cA,N}^g(l)=\zeta_{\cA,N}^{h_1}(l,k_1,k_2)+\zeta_{\cA,N}^{h_2}(k_1,l,k_2)+\zeta_{\cA,N}^{h_3}(k_1,k_2,l)\\
+\zeta_{\cA,N}^{h_4}(k_1+l,k_2)+\zeta_{\cA,N}^{h_5}(k_1,k_2+l),
\end{multline*}
where for each $\alpha\in(\ZZ/N\ZZ)^{\times}$, if $f(\alpha)$ or $g(\alpha)$ equals $\boxtimes$, we set $h_1(\alpha)=\cdots =h_5(\alpha)\coloneqq\boxtimes$; otherwise, writing $f(\alpha)=(f_1(\alpha), f_2(\alpha))\in (\ZZ/N\ZZ)^2$, we define
\begin{align*}
&h_1(\alpha)\coloneqq(g(\alpha),f_1(\alpha),f_2(\alpha)), h_2(\alpha)\coloneqq(f_1(\alpha),g(\alpha),f_2(\alpha)),
h_3(\alpha)\coloneqq(f_1(\alpha),f_2(\alpha),g(\alpha)),\\
&h_4(\alpha)\coloneqq\begin{cases}f(\alpha) & \text{if } f_1(\alpha)=g(\alpha), \\ \boxtimes & \text{otherwise}\end{cases}, h_5(\alpha)\coloneqq\begin{cases}f(\alpha) & \text{if } f_2(\alpha)=g(\alpha), \\ \boxtimes & \text{otherwise.}\end{cases}
\end{align*}

Not only the harmonic product formula, but we also provide natural generalizations of the reversal formula for finite multiple zeta values (Hoffman~\cite[Theorem~4.5]{Hoffman}, Zhao~\cite[Lemma~3.3]{Zhao}; see also \cite[Proposition~2.6]{Saito}) and a formula
\[
\zeta^{}_{\cA}(k_1,\dots,k_r)=\sum_{i=0}^r(-1)^{k_{i+1}+\cdots+k_r}\zeta_{\cA}^{(2)}(k_1,\dots,k_i)\zeta_{\cA}^{(2)}(k_r,\dots,k_{i+1})
\]
obtained by Kaneko, Murakami, and Yoshihara~\cite[(6)]{Kaneko-Murakami-Yoshihara}.
\begin{lemma}\label{lem:j-sum}
Let $\bk=(k_1,\dots,k_r)$ be an index of weight $k$.
Let $j$ be an integer satisfying $0\leq j<N$.
Then, 
\[
\left(\sum_{\frac{jp}{N}<m_1<\cdots<m_r<\frac{(j+1)p}{N}}\frac{1}{m_1^{k_1}\cdots m_r^{k_r}}\bmod{p}\right)_{p\in\PP}=N^k\zeta_{\cA,N}^{[j]}(\bk)
\]
holds in $\cA$.
\end{lemma}
\begin{proof}
Take $\alpha\in(\ZZ/N\ZZ)^{\times}$ and  a prime $p$ belonging to $\alpha$.
By putting $m_i'=Nm_i$, we have
\[
\sum_{\frac{jp}{N}<m_1<\cdots<m_r<\frac{(j+1)p}{N}}\frac{1}{m_1^{k_1}\cdots m_r^{k_r}}=N^k\sum_{\substack{jp<m'_1<\cdots<m'_r<(j+1)p \\ m'_i \ \equiv \ 0 \! \! \pmod{N} \text{ for all } 1\leq i\leq r}}\frac{1}{{m'}_1^{k_1}\cdots{m'}_r^{k_r}}
\]
and by putting $n_i=m'_i-jp$, we have
\[
\sum_{\substack{jp<m'_1<\cdots<m'_r<(j+1)p \\ m'_i \ \equiv \ 0 \! \! \pmod{N} \text{ for all } 1\leq i\leq r}}\frac{1}{{m'}_1^{k_1}\cdots{m'}_r^{k_r}}\equiv \sum_{\substack{0<n_1<\cdots<n_r<p \\ n_i\in-j\alpha \text{ for all } 1\leq i\leq r}}\frac{1}{{n}_1^{k_1}\cdots{n}_r^{k_r}}\pmod{p}.
\]
This gives a proof.
\end{proof}
\begin{proposition}\label{prop:reversal}
Let $\bk=(k_1,\dots,k_r)$ be an index of weight $k$.
\begin{enumerate}[$i)$]
\item Let $\overleftarrow{\bk}\coloneqq(k_r,\dots,k_1)$.
For a color map $c$, we define $\overleftarrow{c}$ as follows$:$ When $c(\alpha)=(\alpha_1,\dots,\alpha_r)\in(\ZZ/N\ZZ)^r$, set $\overleftarrow{c}(\alpha)\coloneqq(\alpha-\alpha_r,\dots,\alpha-\alpha_1)$, and when $c(\alpha)=\boxtimes$, set $\overleftarrow{c}(\alpha)\coloneqq\boxtimes$.
Then, we have
\[
\zeta_{\cA,N}^{c}(\bk)=(-1)^{k}\zeta_{\cA,N}^{\overleftarrow{c}}(\overleftarrow{\bk}).
\]
\item
For each tuple of integers $(i_1,\dots,i_{N-1})$ satisfying $0\leq i_1\leq\cdots\leq i_{N-1}\leq r$ and each integer $j$ satisfying $0\leq j<N$, set $\bk_{(i_1,\dots,i_{N-1});j}\coloneqq(k_{i_j+1},\dots,k_{i_{j+1}})$.
Here, we set $i_0\coloneqq0$, $i_N\coloneqq r$, and if $i_j = i_{j+1}$, then $\bk_{(i_1, \dots, i_{N-1});j} = \emp$.
Then, we have
\[
\zeta^{}_{\cA}(\bk)=N^k\sum_{0\leq i_1\leq\cdots\leq i_{N-1}\leq r}\prod_{j=0}^{N-1}\zeta_{\cA,N}^{[j]}(\bk_{(i_1,\dots,i_{N-1});j}).
\]
\end{enumerate}
\end{proposition}
\begin{proof}
The first formula can be proved using a standard substitution trick $n_i=p-m_{r+1-i}$.
For the second one, it suffices to use the partitioning of the sum range
\[
\sum_{0<m_1<\cdots<m_r<p}=\sum_{0\leq i_1\leq\cdots\leq i_{N-1}\leq r}\sum_{\substack{0<m_1<\cdots<m_{i_1}<\frac{p}{N} \\ \frac{p}{N}<m_{i_1+1}<\cdots<m_{i_2}<\frac{2p}{N} \\ \cdots \\ \frac{(N-1)p}{N}<m_{i_{N-1}+1}<\cdots<m_r<p }}
\]
and Lemma~\ref{lem:j-sum}.
Note that if a prime $p$ does not divide $N$, then for any $j = 1, 2,\dots , N-1$, $jp/N$ is not an integer, and therefore, the above partitioning holds true for all but a finite number of primes $p$.
\end{proof}
The following basic fact will be used later in the proof of Proposition~\ref{prop:submain}.
\begin{lemma}\label{lem:levels}
Let $N$ be a positive integer and $M$ a positive multiple of $N$.
Then, for any non-negative integer $k$, $\cZ_{\cA,k}(N)\subset\cZ_{\cA,k}(M)$ holds.
\end{lemma}
\begin{proof}
This easily follows from the following partitioning of the sum range:
\[
\sum_{\substack{0<m_1<\cdots<m_r<p \\ m_i\equiv a_i \! \! \pmod{N} \text{ for all } 1\leq i\leq r}}=\sum_{0\leq j_1,\dots, j_r<\frac{M}{N}}\sum_{\substack{0<m_1<\cdots<m_r<p \\ m_i\equiv a_i+j_iN \! \! \pmod{N} \text{ for all } 1\leq i\leq r}},
\]
where $p$ is a prime number and $a_1,\dots, a_r$ are integers.
\end{proof}
While investigating $\QQ$-linear relations among finite multiple zeta values of level $N$ seems to be an interesting subject, we will not pursue it further in this note.
\subsection{Non-zeroness}
Researchers of finite multiple zeta values face a significant issue.
Namely, not a single finite multiple zeta value corresponding to a non-empty index has been proved to be non-zero.
In other words, the relation \eqref{eq:example-relation} might imply
\[
2\times 0+0+0+0=0.
\]
There is a possibility that all linear relations that researchers of finite multiple zeta values have proved could be trivial!
(It is crucial to note, however, that this only applies when viewing the relations in $\cA$; in most cases, they have indeed shown non-trivial results.)
This is an extremely frustrating situation, but based on the initial conjectures, the following can be said.
\begin{proposition}\label{prop:main}
\begin{enumerate}[$i)$]
\item\label{item:main1}If Conjecture~$\ref{conj:regular}$ is true, then for every positive integer $k$ that is not equal to $1$, $2$, or $4$, there exists at least one non-zero finite multiple zeta value of weight $k$.
\item\label{item:main2}If Conjecture~$\ref{conj:non-Wieferich}$ is true, then for every even positive integer $N$ and every positive integer $k$, there exists at least one non-zero finite multiple zeta value of level $N$ and weight $k$.
\end{enumerate}
\end{proposition}
It is known that $\cZ_{\cA,1}=\cZ_{\cA,2}=\cZ_{\cA,4}={0}$ (see \cite{Saito} for details).
The proofs of $\ref{item:main1})$ and $\ref{item:main2})$ will be given in Section~\ref{sec:Bernoulli} and Section~\ref{sec:Fermat}, respectively.
For a result regarding odd levels greater than two see Proposition~\ref{prop:submain}.
%%%%%%%%%%%%%%%%%%%%%%%%%%%%%%%%%%%%%%%%
\section{Bernoulli--Seki numbers}\label{sec:Bernoulli}
For a non-negative integer $n$, $B_n$ denotes the $n$th \emph{Bernoulli--Seki number};
\[
\frac{te^t}{e^t-1}=\sum_{n=0}^{\infty}\frac{B_n}{n!}t^n.
\]
Here, the inclusion of ``Seki'' in the name emphasizes that Takakazu Seki independently discovered this sequence alongside Jacob Bernoulli.
Kummer's criterion asserts that $p\geq 5$ is regular if and only if $p$ does not divide any of $B_2$, $B_4$, $\dots$, $B_{p-3}$.
Thus, Conjecture~\ref{conj:regular} implies the following.
\begin{conjecture}\label{conj:non-Wolstenholme}
Let $k\geq 3$ be an odd integer.
Then, there exist infinitely many primes $p$ greater than $k$ such that $p$ does not divide $B_{p-k}$.
\end{conjecture}
Let $k\geq 2$ be an integer, and define an element of $\cA$ from the Bernoulli--Seki numbers as follows:
\[
\fZ(k)\coloneqq\left(\frac{B_{p-k}}{k}\bmod{p}\right)_{p\in\PP}.
\]
If $k$ is even, then $\fZ(k)=0$.
If Conjecture~\ref{conj:non-Wolstenholme} is true, then for odd $k\geq 3$, we have $\fZ(k)\neq0$.
Here, regarding the special values of finite multiple zeta values, the following is well-known (Vandiver, Hoffman, and Zhao~\cite[Theorem~6.1]{Hoffman}, \cite[Theorem~3.1]{Zhao}; see also \cite[(7.2)]{Kaneko}, \cite[Proposition~2.3]{Saito}):
\begin{equation}\label{eq:VHZ}
\zeta^{}_{\cA}(k_1,k_2)=(-1)^{k_2}\binom{k_1+k_2}{k_1}\fZ(k_1+k_2), \quad k_1, k_2\geq 1.
\end{equation}
Thus, the validity of Proposition~\ref{prop:main}~$\ref{item:main1})$ for the case where the weight is an odd integer greater than one is established.
This has been pointed out in several references and is likely well-known (cf.~\cite{Kaneko1}).
In the following, we will discuss the case where the weight is an even integer greater than four.
\begin{proof}[Proof of Proposition~$\ref{prop:main}~\ref{item:main1})$]
Let $k\geq 6$ be an even integer.
In this case, there exist odd integers $k_1$ and $k_2$, both greater than one, such that we can decompose $k$ as $k=k_1+k_2$.
If $p$ is a regular prime greater than both $k_1$ and $k_2$, then by Kummer's criterion, $p$ does not divide either $B_{p-k_1}$ or $B_{p-k_2}$.
Hence, if Conjecture~\ref{conj:regular} is true, we have $\fZ(k_1)\fZ(k_2)\neq0$.
From \eqref{eq:VHZ}, $\fZ(k_1)\in\cZ_{\cA,k_1}$ and $\fZ(k_2)\in\cZ_{\cA,k_2}$.
By the harmonic product formula, $\fZ(k_1)\fZ(k_2)\in\cZ_{\cA,k}$, and this product can be expressed as a $\QQ$-linear combination of finite multiple zeta values of weight $k$.
Therefore, at least one finite multiple zeta value of weight $k$ must be non-zero.
\end{proof}
Although not used later, we provide here two curious formulas for values of level twelve.
\begin{proposition}\label{prop:level12}
Let $k\geq 3$ be an odd integer.
Then, we have
\begin{align*}
2\zeta_{\cA,12}^{[2]}(k)&=(2^{-k}-3^{-k}-4^{-k}+12^{-k})\fZ(k),\\
2\zeta_{\cA,12}^{[3]}(k)&=(3^{-k}-4^{-k}-6^{-k}+12^{-k})\fZ(k).
\end{align*}
\end{proposition}
\begin{proof}
Let $n$ be a positive integer and $p\geq 5$ a prime such that $p-1\nmid 2n$. 
From Vandiver's congruence \cite[(7)]{Vandiver}, we can derive the following two well-known congruences:
\begin{align*}
\frac{3^{p-2n}+4^{p-2n}-6^{p-2n}-1}{4n}B_{2n}&\equiv\sum_{\frac{p}{6}<m<\frac{p}{4}}m^{2n-1}\pmod{p},\\
\frac{2^{p-2n}+3^{p-2n}-4^{p-2n}-1}{4n}B_{2n}&\equiv\sum_{\frac{p}{4}<m<\frac{p}{3}}m^{2n-1}\pmod{p}.
\end{align*}
The first one is also given in \cite[(9)]{Stafford-Vandiver}.
Setting $2n=p-k$ and varying $p$, we obtain the desired formulas by Lemma~\ref{lem:j-sum}.
\end{proof}
%%%%%%%%%%%%%%%%%%%%%%%%%%%%%%%%%%%%%%%%
\section{Fermat quotients}\label{sec:Fermat}
Let $N$ be a positive integer and $p$ a prime number that does not divide $N$.
We define the \emph{Fermat quotient with base} $N$ by
\[
q_p(N)\coloneqq\frac{N^{p-1}-1}{p}.
\]
For an integer $j$ satisfying $0\leq j<N$, we set $s_p(j,N)$ as
\[
s_p(j,N)\coloneqq\sum_{\frac{jp}{N}<m<\frac{(j+1)p}{N}}\frac{1}{m}.
\]
In 1850, Eisenstein~\cite{Eisenstein} proved that $2q_p(2)\equiv-s_p(0,2)\pmod{p}$.
We use the following generalization due to Skula, Dobson, and Ichimura.
\begin{theorem}[Skula~\cite{Skula}, Dobson~\cite{Dobson}, Ichimura~\cite{Ichimura}]\label{thm:SDI}
For every positive integer $N$, we have
\[
(N+1)q_p(2)\equiv-\sum_{0\leq j<\frac{N}{2}}s_p(2j,2N)\pmod{p}.
\]
\end{theorem}
Here, we define the logarithm function $\log_{\cA}\colon\QQ^{\times}\to\cA$ as $\log_{\cA}(N)\coloneqq(q_p(N)\bmod{p})_{p\in\PP}$.
(The Fermat quotient can be defined similarly even when $N$ is not a positive integer.)
The reason for the name is that the logarithmic law can be derived from the congruence $q_p(NM)\equiv q_p(N)+q_p(M)\pmod{p}$. 
Theorem~\ref{thm:SDI} derives
\begin{equation}\label{eq:A-SDI}
\log_{\cA}(2)=-\frac{2N}{N+1}\sum_{0\leq j<N/2}\zeta_{\cA,2N}^{[2j]}(1)
\end{equation}
using Lemma~\ref{lem:j-sum}, and in particular it is shown that $\log_{\cA}(2)\in\cZ_{\cA,1}(2N)$ for every $N$.
\begin{proof}[Proof of Proposition~$\ref{prop:main}~\ref{item:main2})$]
We assume that Conjecture~\ref{conj:non-Wieferich} is true.
By this assumption, there exist infinitely many primes $p$ that do not divide $q_p(2)$ and hence $\log_{\cA}(2)\neq0$.
By \eqref{eq:A-SDI}, there exists an integer $j$ satisfying $0\leq j<N/2$ (fix one such $j$) such that $\zeta_{\cA,2N}^{[2j]}(1)\neq0$.
Let $k$ be an arbitrary positive integer.
Since $\cA$ is a reduced ring, we have $\zeta_{\cA,2N}^{[2j]}(1)^k\neq0$.
Since $\zeta_{\cA,2N}^{[2j]}(1)^k\in\cZ_{\cA,k}(2N)$ and this value can be expressed as a $\QQ$-lenear combination of finite multiple zeta values of level $2N$ and weight $k$, we see that at least one finite multiple zeta value of level $2N$ and weight $k$ must be non-zero.
\end{proof}
By using Lerch's congruence instead of Skula, Dobson, and Ichimura's congruence, we can obtain the following proposition through a similar argument.
\begin{theorem}[{Lerch~\cite[(8)]{Lerch}}]\label{thm:Lerch}
For $N\geq 2$, we have
\[
Nq_p(N)\equiv\sum_{j=1}^{N-1}js_p(j,N)\pmod{p}.
\]
\end{theorem}
This theorem derives
\[
\log_{\cA}(N)=\sum_{j=1}^{N-1}j\cdot\zeta_{\cA,N}^{[j]}(1)\in\cZ_{\cA,1}(N).
\]
In particular, by Lemma~\ref{lem:levels}, we see that $\log_{\cA}(N)\in\cZ_{\cA,1}(M)$ for any positive integer $M$ that is a multiple of $N$.
\begin{proposition}\label{prop:submain}
Let $N\geq 2$ be an integer.
Assume that there exist infinitely many primes $p$ that do not divide $q_p(N)$.
Let $M$ be any positive multiple of $N$.
Then, for each positive integer $k$, there exists at least one non-zero finite multiple zeta value of level $M$ and weight $k$.
\end{proposition}
\begin{remark}
Proposition~\ref{prop:submain} is an extension of Proposition~\ref{prop:main}~$\ref{item:main2}$), and thanks to Lemma \ref{lem:levels}, to obtain Proposition~\ref{prop:main}~$\ref{item:main2}$), Eisenstein's congruence is sufficient, while Skula, Dobson, and Ichimura's congruence is not necessary.
However, \eqref{eq:A-SDI} demonstrates a fact that is stronger than the mere $\log_{\cA}(2)\in\cZ_{\cA,1}(2N)$, making it worth mentioning due to the fact that it restricts the types of color maps used in zeta values to the form of $[2j]$.
\end{remark}
%%%%%%%%%%%%%%%%%%%%%%%%%%%%%%%%%%%%%%%%
\section{Problems}\label{sec:problems}
Conjecture~\ref{conj:regular} has been shown to imply the existence of non-zero zeta values for each weight.
However, in the current state where not a single non-zero value is known to exist, it is crucial to prove the existence of at least one.
To achieve this, it is sufficient to address the following problem, which is weaker than Conjecture~\ref{conj:non-Wolstenholme}.
\begin{problem}\label{prob:1}
Show that there exists at least one odd integer $k\geq 3$ such that there are infinitely many primes $p>k$ that do not divide $B_{p-k}$.
\end{problem}
If this problem is resolved in the affirmative, it implies that at least one non-zero finite multiple zeta value exists.
From the perspective of the Kaneko--Zagier conjecture, $\fZ(k)$ can be considered as a counterpart to the Riemann zeta value $\zeta(k)$.
In a sense, this problem aims for a result analogous to Ball--Rivoal~\cite{Bal-Rivoal} and Zudilin~\cite{Zudilin}'s outstanding works for Riemann zeta values, albeit in a slightly weaker form.
For every odd integer $k\geq 3$, it is conjectured that $\fZ(k)$ is not only non-zero but also irrational.
More strongly, it is expected to be a non-zero zero divisor.

Regarding Conjecture~$\ref{conj:non-Wieferich}$ and the assumption of Proposition~$\ref{prop:submain}$, several conditional results are known.
One of the most famous results is due to Silverman~\cite{Silverman}, asserts the correctness of these conjectures under the ABC conjecture.
It should be noted that the author lacks the expertise to comprehend the claimed proof of the ABC conjecture by Mochizuki \cite{Mochizuki}.
On a related note, it's worth mentioning about Artin's primitive root conjecture, which is a well-known conjecture about a different kind of prime numbers.
It states that for any integer $g$ that is neither $-1$ nor a square integer, there exist infinitely many primes $p$ for which $g$ is a primitive root modulo $p$.
This conjecture was proved for all values of $g$ by Hooley~\cite{Hooley} under the generalized Riemann hypothesis.
On the other hand, there exist unconditional results which state that the conjecture holds for at least one value of $g$, as shown by Gupta--Ram Murty~\cite{Gupta-Murty} and Heath-Brown~\cite{Heath-Brown}.
In particular, according to Heath-Brown's result, Artin's conjecture is true for at least one of the values $g=2$, $3$, or $5$.
Expecting a similar progress for the case of Fermat quotients, an unconditional solution to the following problem is desired.
\begin{problem}\label{prob:2}
Show that there exists at least one integer $N\geq 2$ such that there are infinitely many primes $p$ which do not divide $q_p(N)$.
\end{problem}
If this problem is resolved in the affirmative, it implies that for at least one (therefore infinitely many) $N\geq 2$, there exists a non-zero finite multiple zeta value of level $N$ for each weight.

For each odd prime $p$, $\ell_p$ denotes the least positive integer $N$ for which $p$ does not divide $q_p(N)$.
There have been several studies providing upper bounds for $\ell_p$.
An early result by Lenstra Jr.~\cite{Lenstra} gives $\ell_p\leq 4(\log p)^2$.
One of more recent improvements is $\ell_p\leq(\log p)^{463/252+o(1)}$ as $p\to\infty$, proved by Bourgain, Ford, Konyagin, and Shparlinski~\cite{Bourgain-etal}.
We aim to eliminate the dependence on $p$ in these upper bounds.
\begin{problem}[{\cite[Conjecture~9]{Granville}}]\label{prob:3}
Show that there exists an integer $M$ such that $\ell_p\leq M$ holds for every odd prime $p$.
\end{problem}
If Problem~\ref{prob:3} is resolved affirmatively, it is clear that Problem~\ref{prob:2} follows in the affirmative as well.
Let us provide an alternative formulation of the statement in Problem~\ref{prob:3}.
For each $N$, we set
\[
W(N)\coloneqq\{p\in\PP : N^{p-1}-1\equiv0\pmod{p^2}\}, \quad W^{\rc}(N)\coloneqq\PP\setminus W(N).
\]
\begin{lemma}\label{lem:l_p}
A positive resolution of Problem~$\ref{prob:3}$ is equivalent to the validity of the following statement$:$ there exists an integer $M\geq 3$ such that the intersection $\bigcap_{N=2}^MW(N)$ is finite.
\end{lemma}
\begin{proof}
Let $M\geq 3$ be an integer and we note that $\ell_p<p$.
We can easily check that
\[
\forall p\geq 3, \ell_p\leq M \Longleftrightarrow \PP=\bigcup_{N=2}^MW^{\rc}(N) \Longleftrightarrow \bigcap_{N=2}^MW(N)=\emp.
\]
Assume that $\bigcap_{N=2}^MW(N)$ is a non-empty finite set and write it $\{p_1, \dots, p_k\}$ by enumeration.
Let $L$ be defined as $L\coloneqq\max\{\ell_{p_1},\dots, \ell_{p_k}\}$.
Then, we have $\bigcap_{N=2}^LW(N)=\emp$.
\end{proof}
Let us define $\eth_p$ as the Bernoulli--Seki number analogue of $\ell_p$.
\begin{definition}\label{def:d_p}
Let $p\geq 5$ be a prime.
We define $\eth_p$ as the least odd integer $k\geq 3$ for which $p$ does not divide $B_{p-k}$.
\end{definition}
\begin{proposition}\label{prop:upper_bound_eth}
For every prime $p\geq11$, we have
\[
\eth_p\leq\begin{cases}\frac{p-3}{2} & \text{if } p\equiv1\pmod{4}, \\ \frac{p-5}{2} & \text{if } p\equiv3\pmod{4}.\end{cases}
\]
\end{proposition}
\begin{proof}
Let $i(p)$ denote the index of irregularity defined by $i(p)\coloneqq\#\{2k : p\mid B_{2k}, 2\leq 2k\leq p-3\}$.
Since a trivial inequality $\eth_p\leq 2i(p)+3$ holds, it suffices to provide an appropriate upper bound for $i(p)$.
Let $h_p^-$ denote the first factor of the class number of the $p$th cyclotomic field.
We then have the inequality $i(p)\leq\log(h_p^-)/\log p$ from Vandiver's result~\cite{Vandiver1}.
Since an elementary inequality $p^{\frac{p+3}{4}}2^{-\frac{p-1}{4}}\leq p^{\frac{p-8}{4}}$ holds for $p\geq 71$, we can obtain the desired bound from Carlitz's inequality $h_p^-<p^{\frac{p+3}{4}}2^{-\frac{p-1}{4}}$ (\cite[(21)]{Carlitz}).
Here, note that $\eth_p=3$ for all $p<16843$.
\end{proof}
\begin{remark}
When $p\equiv3\pmod{4}$, it is known from the congruence due to Cauchy (see \cite[5.2]{Carlitz1}) that $B_{\frac{p+1}{2}}\not\equiv0\pmod{p}$, leading to $\eth_p\leq\frac{p-1}{2}$.
When $p\equiv1\pmod{4}$, it is conjectured that $B_{\frac{p-1}{2}}\not\equiv0\pmod{p}$ by the congruence due to Kiselev~\cite{Kiselev} and the Ankeny, Artin, and Chowla conjecture~\cite{Ankeny-Artin-Chowla}, leading to $\eth_p\leq\frac{p+1}{2}$.
While the proposition above provides better bounds than those derived from these considerations, it would be desirable to obtain more non-trivial bound, for example, $\eth_p \leq p^{1-\varepsilon}$ for some positive $\varepsilon$.
\end{remark}
It is conjectured that $\sup_{p\in\PP}i(p)=\infty$, but on the other hand, $\eth_p$ is expected to be bounded. (Does $\eth_p\leq 5$ always hold?)
\begin{problem}\label{prob:4}
Show that there exists an integer $M$ such that $\eth_p\leq M$ holds for every prime $p\geq 5$.
\end{problem}
If Problem~\ref{prob:4} is resolved affirmatively, it is clear that Problem~\ref{prob:1} follows in the affirmative as well.
For each odd integer $k\geq 3$, we set $I(k)\coloneqq\{p\in\PP : p\mid B_{p-k}\}$.
By a similar argument as in the proof of  Lemma~\ref{lem:l_p}, we can prove the following.
\begin{lemma}\label{lem:d_p}
A positive resolution of Problem~$\ref{prob:4}$ is equivalent to the validity of the following statement$:$ there exists an integer $M\geq 5$ such that the intersection $\bigcap_{3\leq k\leq M, \ k: \text{ odd}}I(k)$ is finite.
\end{lemma}
Conjecture~\ref{conj:regular} remains open; however, the infinitude of irregular primes has been proved (Jensen~\cite{Jensen}, Carlitz~\cite{Carlitz2}).
Furthermore, the following lower bound has been obtained by Luca, Pizarro-Madariaga, and Pomerance~\cite{Luca-etal}:
\[
\#\{p\leq x : p \text{ is an irregular prime}\}\geq(1+o(1))\frac{\log\log x}{\log\log\log x},\quad \text{as } x\to\infty.
\]
The relative density of the set of irregular primes in the set of all prime numbers is conjectured to be $1-e^{-\frac{1}{2}}$.
On the contrary, the relative density of the set of Wieferich primes in the set of all prime numbers is conjectured to be zero, and proving the infinitude of Wieferich primes may be highly challenging.
If this conjecture holds true, then $\log_{\cA}(2)$ would become a zero divisor.
%%%%%%%%%%%%%%%%%%%%%%%%%%%%%%%%%%%%%%%%
\section*{Acknowledgements}
The author thanks Professors Masanobu Kaneko, Toshiki Matsusaka, Pieter Moree, Masataka Ono, and Shingo Saito for their valuable comments.
The author also grateful to Hanamichi Kawamura for discussions.
%%%%%%%%%%%%%%%%%%%%%%%%%%%%%%%%%%%%%%%%

\end{document}